\documentclass[10pt]{article}

\usepackage{ucs}
\usepackage{amssymb}
\usepackage{amsthm}
\usepackage{amsmath}
\usepackage{latexsym}
\usepackage[cp1251]{inputenc}
\usepackage[english]{babel}
\usepackage{graphicx}
\usepackage[justification=centering, labelfont=bf]{caption}
\usepackage{txfonts}
\usepackage{lmodern}
\usepackage{mathrsfs}
\usepackage{eufrak}
\usepackage{dsfont}
\usepackage{enumerate}
\usepackage{hyphenat}
\usepackage[hidelinks]{hyperref}
\usepackage{mathtools}
\usepackage{bbold}
\usepackage{mdframed}
\usepackage{tikz-cd}
\usepackage{tikz}
\usepackage{amsbsy}
\usepackage[left=2cm,right=2cm,top=1.9cm,bottom=2.4cm,bindingoffset=0cm]{geometry}

\usetikzlibrary{arrows}
\usetikzlibrary{shapes}

\numberwithin{equation}{subsection}

\title{The Local Cut Lemma}
\date{}
\author{Anton Bernshteyn\thanks{Department of Mathematics, University of Illinois at Urbana--Champaign, IL, USA, \texttt{bernsht2@illinois.edu}. This research is supported by the Illinois Distinguished Fellowship.}}

\newtheorem{theo}{Theorem}[section]

\newtheorem{lemma}[theo]{Lemma}

\newtheorem{corl}[theo]{Corollary}
\newtheorem{conj}[theo]{Conjecture}
\newtheorem{claim}[theo]{Claim}
\newtheorem{small_claim}{Claim}[theo]

\theoremstyle{definition}
\newtheorem{defn}[theo]{Definition}

\newtheorem{remk}[theo]{Remark}

\newcommand*{\myproofname}{Proof}
\newenvironment{claimproof}[1][\myproofname]{\begin{proof}[#1]}{\end{proof}}

\newcommand{\powerset}[1]{\operatorname{Pow}(#1)}

\begin{document}	
	\maketitle
	
	\begin{abstract}
		The Lov\'{a}sz Local Lemma is a very powerful tool in probabilistic combinatorics, that is often used to prove existence of combinatorial objects satisfying certain constraints. Moser and Tardos \cite{Moser} have shown that the LLL gives more than just pure existence results: there is an effective randomized algorithm that can be used to find a desired object. In order to analyze this algorithm, Moser and Tardos developed the so-called \emph{entropy compression method}. It turned out that one could obtain better combinatorial results by a direct application of the entropy compression method rather than simply appealing to the LLL. The aim of this paper is to provide a generalization of the LLL which implies these new combinatorial results. This generalization, which we call the Local Cut Lemma, concerns a random cut in a directed graph with certain properties. Note that our result has a short probabilistic proof that does not use entropy compression. As a consequence, it not only shows that a certain probability is positive, but also gives an explicit lower bound for this probability. As an illustration, we present a new application (an improved lower bound on the number of edges in color-critical hypergraphs) as well as explain how to use the Local Cut Lemma to derive some of the results obtained previously using the entropy compression method.
	\end{abstract}
	
	
	\tableofcontents
	
	\section{Introduction}
	
	One of the most useful tools in probabilistic combinatorics is the so-called \emph{Lov\'{a}sz Local Lemma} (the LLL for short), which was proved by Erd\H{o}s and Lov\'{a}sz in their seminal paper~\cite{EL}. Roughly speaking, the LLL asserts that, given a family $\mathcal{B}$ of random events whose individual probabilities are small and whose dependency is somehow limited, there is a positive probability that none of the events in $\mathcal{B}$ happen. More precisely:
	
	\begin{theo}[Lov\'{a}sz Local Lemma, \cite{AS}]
		Let $B_1$, \ldots, $B_n$ be random events in a probability space $\Omega$. For each $1 \leq i \leq n$, let $\Gamma(i)$ be a subset of $\{1, \ldots, n\} \setminus \{i\}$ such that the event $B_i$ is independent from the algebra generated by the events $B_j$ with $j \not \in \Gamma(i) \cup \{i\}$. Suppose that there exists a function $\mu\colon \{1, \ldots, n\}\to[0;1)$ such that for every $1 \leq i \leq n$,
		$$
		\Pr(B_i)\leq \mu(i) \prod_{j \in \Gamma(i)}(1-\mu(j)).
		$$
		Then
		$$
		\Pr\left(\bigcap_{i =1}^n\overline{B_i}\right) \geq \prod_{i = 1}^n(1 - \mu(i)) > 0.
		$$
	\end{theo}
	
	Note that the probability $\Pr\left(\bigcap_{i \in I}\overline{B_i}\right)$, which the LLL bounds from below, is usually exponentially small (in the parameter $n$). This is in contrast to the more common situation in the probabilistic method when the probability of interest is not only positive, but separated from zero. Although this property of the LLL makes it an indispensable tool in proving combinatorial existence results, it also makes these results seemingly nonconstructive, since sampling the probability space to find an object with the desired properties would usually take an exponentially long expected time. A major breakthrough was made by Moser and Tardos~\cite{Moser}, who showed that, in a special framework for the LLL called the \emph{variable version} (the name is due to Kolipaka and Szegedy~\cite{Kolipaka1}), there exists a simple Las Vegas algorithm with expected polynomial runtime that searches the probability space for a point which avoids all the events in $\mathcal{B}$. Their algorithm was subsequently refined and extended to other situations by several authors; see e.g.~\cite{Pegden}, \cite{Kolipaka1}, \cite{ach}, \cite{CGH}.
	
	The key ingredient of Moser and Tardos's proof is the so-called \emph{entropy compression method} (the name is due to Tao~\cite{Tao}). The idea of this method is to encode the execution process of the algorithm in such a way that the original sequence of random inputs can be uniquely recovered from the resulting encoding. One then proceeds to show that if the algorithm runs for too long, the space of possible codes becomes smaller than the space of inputs, which leads to a contradiction.
	
	It was discovered lately (and somewhat unexpectedly) that applying the entropy compression method directly can often produce better combinatorial results than simply using the LLL. The idea, first introduced by Grytczuk, Kozik, and Micek in their study of nonrepetitive sequences~\cite{Grytczuk}, is to construct a randomized procedure that solves a given combinatorial problem and then apply the entropy compression argument to show that it runs in expected finite time. A wealth of new results have been obtained using this paradigm; see e.g.~\cite{Duj}, \cite{Esperet}, \cite{Goncalves}. Some of these examples are discussed in more detail in Section~\ref{sec:applications}.
	
	Note that the entropy compression method is indeed a ``method'' that one can use to attack a problem rather than a general theorem that contains various combinatorial results as its special cases. It is natural to ask if such a theorem exists, i.e., if there is a generalization of the LLL that implies the new combinatorial results obtained using the entropy compression method. The goal of this paper is to provide such a generalization, which we call the \emph{Local Cut Lemma} (the LCL for short). It is important to note that this result is purely probabilistic and similar to the LLL in flavor. In particular, its short and simple probabilistic proof does not use the entropy compression method. Instead, it estimates certain probabilities explicitly, in much the same way as the original (nonconstructive) proof of the~LLL does. We state and prove the~LCL in Section~\ref{sec:statement}. Section~\ref{sec:applications} is dedicated to applications of the~LCL. We start by introducing a simplified special case of the~LCL (namely Theorem~\ref{theo:hypercubes}) in Subsection~\ref{subsec:special}, which turns out to be sufficient for most applications. In fact, Theorem~\ref{theo:hypercubes} already implies the classical LLL, as we show in Subsection~\ref{subsec:LLL}. In Subsection~\ref{subsec:hypcol}, we discuss one simple example (namely hypergraph coloring), which provides the intuition behind the~LCL and serves as a model for more substantial applications described later. In Subsections~\ref{subsec:nonrep} and~\ref{subsec:acyclic} we show how to use the~LCL to prove several results obtained previously using the entropy compression method. We also present a new application (an improved lower bound on the number of edges in color-critical hypergraphs) in Subsection~\ref{subsec:critical}. The last application, discussed in Subsection~\ref{subsec:choice}, is a curious probabilistic corollary of the LCL.
	
	\section{The Local Cut Lemma}\label{sec:statement}
	
	\subsection{Statement of the~LCL}
	
	To state our main result, we need to fix some notation and terminology. In what follows, a \emph{digraph} always means a finite 
	directed multigraph. Let $D$ be a digraph with vertex set $V$ and edge set $E$. For $x$, $y \in V$, let $E(x,y) \subseteq E$ denote the set of all edges with tail $x$ and head $y$.
	
	A digraph $D$ is \emph{simple} if for all $x$, $y \in V$, $|E(x,y)| \leq 1$. If $D$ is simple and $|E(x,y)| = 1$, then the unique edge with tail $x$ and head $y$ is denoted by $xy$ (or sometimes $(x,y)$). For an arbitrary digraph $D$, let $D^s$ denote its \emph{underlying simple digraph}, i.e., the simple digraph with vertex set $V$ in which $xy$ is an edge if and only if $E(x,y) \neq \emptyset$. Denote the edge set of $D^s$ by $E^s$. For a set $F \subseteq E$, let $F^s \subseteq E^s$ be the set of all edges $xy \in E^s$ such that $F \cap E(x,y) \neq \emptyset$. A set $A \subseteq V$ is \emph{out-closed} (resp. \emph{in-closed}) if for all $xy \in E^s$, $x \in A$ implies $y \in A$ (resp. $y \in A$ implies $x \in A$).
	
	\begin{defn}
		Let $D$ be a digraph with vertex set $V$ and edge set $E$ and let $A \subseteq V$ be an out-closed set of vertices. A set $F \subseteq E$ of edges is an \emph{$A$-cut} if $A$ is in-closed in $D^s - F^s$. In other words, a set $F \subseteq E$ is an $A$-cut if it contains at least one edge $e \in E(x, y)$ for all $xy \in E^s$ such that $x \not \in A$ and $y \in A$ (see Fig.~\ref{fig:cut}).
	\end{defn}
	
	\begin{figure}[h]
		\begin{mdframed}[linewidth=0.5pt]
			\centering
			\includegraphics{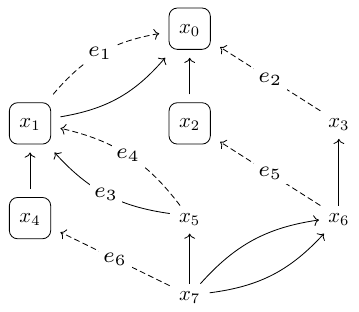}
		\caption{\label{fig:cut} A digraph $D$ with an out-closed set $A = \{x_0, x_1, x_2, x_4\}$. Any $A$-cut must contain the edges $\{e_2, e_5, e_6\}$ and at least one of $\{e_3, e_4\}$. For example, the set $F = \{e_1, e_2, e_4, e_5, e_6\}$ consisting of the dashed edges forms an $A$-cut.}
		\end{mdframed}
	\end{figure}
	
	We say that a vertex $z \in V$ is \emph{reachable} from $x \in V$ if $D$ (or, equivalently, $D^s$) contains a directed $xz$-path. The set of all vertices reachable from $x$ is denoted by $R_D(x)$.
	
	\begin{defn}
		Let $D$ be a digraph with vertex set $V$ and edge set $E$. For a function $\omega \colon E^s \to [1; +\infty)$ and vertices $x \in V$ and $z \in R_D(x)$, define
		$$
			\underline{\omega}(x,z) \coloneqq \min \left\{\prod_{i=1}^k \omega(z_{i-1}z_i) \,:\, \text{$x = z_0\longrightarrow z_1\longrightarrow$\ldots$\longrightarrow z_k = z$ is a directed $xz$-path in $D^s$}\right\}.
		$$
	\end{defn}
	
	For a set $S$, we use $\powerset{S}$ to denote the power set of $S$, i.e., the set of all subsets of $S$.
	
	\begin{defn}\label{defn:risk}
		Let $D$ be a digraph with vertex set $V$ and edge set $E$. Let $\Omega$ be a probability space and let $A \colon \Omega \to \powerset{V}$ and $F \colon \Omega \to \powerset{E}$ be random variables such that with probability~$1$, $A$ is an out-closed set of vertices and $F$ is an $A$-cut. Fix a function $\omega \colon E^s \to [1; +\infty)$. For $xy \in E^s$, $e \in E(x,y)$, and $z \in R_D(y)$, let
		$$
			\rho^{A,F}_\omega(e, z) \coloneqq \Pr(e \in F \vert z \in A) \cdot \underline{\omega}(x, z).
		$$
		For $e \in E(x,y)$, define the \emph{risk} to $e$ as
		$$
			\rho^{A, F}_\omega(e) \coloneqq \min_{z \in R_D(y)} \rho^{A, F}_\omega(e, z).
		$$
	\end{defn}
	
	\begin{remk}\label{remk:cond_prob}
		For random events $P$, $Q$, the conditional probability $\Pr(P\vert Q)$ is only defined if $\Pr(Q) > 0$. For convenience, we adopt the following notational convention in Definition~\ref{defn:risk}: If $Q$ is a random event and $\Pr(Q) = 0$, then $\Pr(P\vert Q) = 0$ for all events $P$. Note that this way the crucial equation $\Pr(P \vert Q) \cdot \Pr(Q) = \Pr(P \cap Q)$ is satisfied even when $\Pr(Q) = 0$, and this is the only property of conditional probability we will use.
	\end{remk}
	
	We are now ready to state the main result of this paper.
	
	\begin{theo}[Local Cut Lemma]\label{theo:main}
		Let $D$ be a digraph with vertex set $V$ and edge set $E$. Let $\Omega$ be a probability space and let $A \colon \Omega \to \powerset{V}$ and $F \colon \Omega \to \powerset{E}$ be random variables such that with probability~$1$, $A$ is an out-closed set of vertices and $F$ is an $A$-cut. If a function $\omega \colon E^s \to [1; +\infty)$ satisfies the following inequality for all $xy \in E^s$:
		\begin{equation}\label{eq:main}
			\omega(xy) \geq 1 + \sum_{e \in E(x,y)} \rho^{A,F}_\omega(e),
		\end{equation}
		then for all $xy \in E^s$,
		$$
			\Pr(y \in A) \leq \Pr(x \in A) \cdot\omega(xy).
		$$
	\end{theo}
	
	The following immediate corollary is the main tool used in combinatorial applications of Theorem~\ref{theo:main}:
	
	\begin{corl}\label{corl:positive}
		Let $D$, $A$, $F$, $\omega$ be as in Theorem~\ref{theo:main}. Let $x\in V$, $z \in R_D(x)$, and suppose that $\Pr(z \in A) > 0$. Then
		$$
		\Pr(x \in A) \geq \frac{\Pr(z \in A)}{\underline{\omega}(x,z)} > 0.
		$$
	\end{corl}
	
	\subsection{Proof of the LCL}\label{subsec:proof}
	
	In this subsection we prove Theorem~\ref{theo:main}. Let $D$, $A$, $F$ be as in the statement of Theorem~\ref{theo:main} and assume that a function $\omega \colon E^s \to [1; +\infty)$ satisfies
	\begin{equation}
	\omega(xy) \geq 1 + \sum_{e \in E(x,y)} \rho^{A, F}_\omega(e)\tag{\ref{eq:main}}
	\end{equation}
	for all $xy \in E^s$. For each $\upsilon \colon E^s \to [1; +\infty)$, let $f(\upsilon) \colon E^s \to [1; +\infty)$ be defined by
	$$
	f(\upsilon)(xy) \coloneqq 1 + \sum_{e \in E(x,y)} \rho^{A,F}_\upsilon(e).
	$$
	Also, let $f(\mathbb{0}) \coloneqq \mathbb{1}$, where $\mathbb{0}$ and $\mathbb{1}$ denote the constant $0$ and $1$ functions respectively. Then \eqref{eq:main} is equivalent to
	\begin{equation}\label{eq:main1}
	\omega(xy) \geq f(\omega)(xy).
	\end{equation}
	Note that the map $f$ is monotone increasing, i.e., if $\upsilon(xy) \leq \upsilon'(xy)$ for all $xy \in E^s$, then $f(\upsilon)(xy) \leq f(\upsilon')(xy)$ for all $xy \in E^s$ as well.
	
	Let $\omega_0 \coloneqq \mathbb{0}$ and let $\omega_{n+1} \coloneqq f(\omega_n)$ for all $n \in \mathbb{Z}_{\geq 0}$. To simplify the notation, let $\rho_n \coloneqq \rho^{A, F}_{\omega_n}$.
	
	\begin{claim}
		For all $n \in \mathbb{Z}_{\geq 0}$ and $xy \in E^s$,
		\begin{equation}\label{eq:monotone}
		\omega_n(xy) \leq \omega_{n+1}(xy).
		\end{equation}
	\end{claim}
	\begin{proof}
		Proof is by induction on $n$. If $n=0$, then \eqref{eq:monotone} asserts that $0 \leq 1$. Now suppose that \eqref{eq:monotone} holds for some $n \in \mathbb{Z}_{\geq 0}$. Then we have
		$$
		\omega_{n+1}(xy) = f(\omega_n)(xy) \leq f(\omega_{n+1})(xy) = \omega_{n+2}(xy),
		$$
		as desired.
	\end{proof}
	
	\begin{claim}
		For all $n \in \mathbb{Z}_{\geq 0}$ and $xy \in E^s$,
		\begin{equation}\label{eq:bounded}
		\omega_n(xy) \leq \omega(xy).
		\end{equation}
	\end{claim}
	\begin{proof}
		Proof is again by induction on $n$. If $n = 0$, then \eqref{eq:bounded} says that $0 \leq \omega(xy)$. Now suppose that \eqref{eq:bounded} holds for some $n \in \mathbb{Z}_{\geq 0}$. Then, using \eqref{eq:main1}, we get
		$$
		\omega_{n+1}(xy) = f(\omega_n)(xy) \leq f(\omega)(xy) \leq \omega(xy),
		$$
		as desired.
	\end{proof}
	
	Since the sequence $\{\omega_n(xy)\}_{n =0}^\infty$ is monotone increasing and bounded by $\omega(xy)$, it has a limit, so let
	$$
	\omega_\infty(xy) \coloneqq\lim_{n \to \infty} \omega_n(xy).
	$$
	Note that we still have $\omega_\infty(xy) \leq \omega(xy)$ for all $xy \in E^s$. Hence it is enough to prove that for all $xy \in E^s$,
	\begin{equation}\label{eq:infty}
	\Pr(y \in A) \leq \Pr(x \in A) \cdot \omega_\infty(xy).
	\end{equation}
	We will derive \eqref{eq:infty} from the following lemma.
	
	\begin{lemma}\label{lemma:n}
		For every $n \in \mathbb{Z}_{\geq 0}$ and $xy \in E^s$,
		\begin{equation}\label{eq:n}
		\Pr(y \in A) \leq \Pr(x \in A) \cdot \omega_n(xy) + \omega_{n+1}(xy)-\omega_n(xy).
		\end{equation}
	\end{lemma}
	
	If Lemma~\ref{lemma:n} holds, then we are done, since it implies that
	$$
	\Pr(y \in A) \leq \lim_{n \to \infty} \left(\Pr(x \in A) \cdot \omega_n(xy) + \omega_{n+1}(xy)-\omega_n(xy)\right) = \Pr(x \in A) \cdot \omega_\infty(xy),
	$$
	as desired.
	
	To establish Lemma~\ref{lemma:n}, we need the following claim.
	
	\begin{claim}\label{claim:generalizedn}
		Let $n \in \mathbb{Z}_{\geq 0}$ and suppose that for all $xy \in E^s$, \eqref{eq:n} holds. Then for all $x \in V$ and $z \in R_D(x)$,
		\begin{equation}\label{eq:generalizedn}
		\Pr(z \in A) \leq \Pr(x \in A) \cdot \underline{\omega_n}(x,z) + \underline{\omega_{n+1}}(x,z)-\underline{\omega_n}(x,z).
		\end{equation}
	\end{claim}
	
	The proof of Claim~\ref{claim:generalizedn} uses the following simple algebraic inequality.
	
	\begin{claim}\label{claim:algineq}
		Let $a_1$, \ldots, $a_k$, $b_1$, \ldots, $b_k$ be nonnegative real numbers with $b_i \geq \max\{a_i,1\}$ for all $1 \leq i \leq k$. Then
		\begin{equation}\label{eq:algineq}
		\sum_{i = 1}^k \left(\prod_{j=1}^{i-1}a_j\right) (b_i - a_i) \,\leq\, \prod_{i=1}^k b_i - \prod_{i=1}^k a_i.
		\end{equation}		
	\end{claim}
	\begin{proof}
		Proof is by induction on $k$. If $k = 1$, then both sides of \eqref{eq:algineq} are equal to $b_1 -a_1$. If the claim is established for some $k$, then for $k+1$ we get
		\begin{align*}
		\sum_{i = 1}^{k+1} \left(\prod_{j=1}^{i-1}a_j\right) (b_i - a_i) \,&=\, \sum_{i = 1}^k \left(\prod_{j=1}^{i-1}a_j\right) (b_i - a_i) + \left(\prod_{i=1}^k a_i\right)b_{k+1} - \prod_{i=1}^{k+1} a_i\\
		&\leq\, \prod_{i=1}^k b_i -\prod_{i=1}^k a_i + \left(\prod_{i=1}^k a_i\right)b_{k+1} - \prod_{i=1}^{k+1} a_i \\
		&=\,\prod_{i=1}^{k+1}b_i - \prod_{i=1}^{k+1}a_i -\left(\prod_{i=1}^k b_i - \prod_{i=1}^k a_i\right)(b_{k+1}-1)\\
		&\leq\, \prod_{i=1}^{k+1}b_i - \prod_{i=1}^{k+1}a_i,
		\end{align*}
		as desired.
	\end{proof}
	
	\begin{proof}[Proof of Claim~\ref{claim:generalizedn}]
		Let $x = z_0\longrightarrow z_1\longrightarrow$\ldots$\longrightarrow z_k = z$ be some directed $xz$-path in $D^s$. For $1 \leq i \leq k$, let $a_i \coloneqq \omega_n(z_{k-i} z_{k-i+1})$ and $b_i \coloneqq \omega_{n+1}(z_{k-i} z_{k-i+1})$. Note that $b_i \geq \max\{a_i, 1\}$.
		
		Due to \eqref{eq:n}, we have
		$$
		\Pr(z \in A) \leq \Pr(z_{k-1} \in A) \cdot a_1 + b_1 - a_1.
		$$
		Similarly,
		$$
		\Pr(z_{k-1} \in A) \leq \Pr(z_{k-2} \in A) \cdot a_2 + b_2 - a_2,
		$$
		so
		$$
		\Pr(z \in A) \leq \Pr(z_{k-2} \in A) \cdot a_1 a_2 + b_1 - a_1 + a_1(b_2-a_2).
		$$
		Continuing such substitutions, we finally obtain
		$$
		\Pr(z \in A) \leq \Pr(x \in A) \cdot \prod_{i = 1}^k a_i + \sum_{i=1}^{k} \left(\prod_{j=1}^{i-1}a_j\right)(b_i - a_i).
		$$
		Using Claim~\ref{claim:algineq}, we get
		$$
		\Pr(z \in A) \leq \Pr(x \in A) \cdot \prod_{i = 1}^k a_i + \prod_{i = 1}^k b_i - \prod_{i=1}^k a_i.
		$$
		Note that
		$$
		\prod_{i=1}^k a_i = \prod_{i=1}^k \omega_n(z_{i-1}z_i) \geq \underline{\omega_n}(x,z).
		$$
		Since $\Pr(x \in A) \leq 1$, this implies
		\begin{equation}\label{eq:bs}
		\Pr(z \in A) \leq \Pr(x \in A) \cdot \underline{\omega_n}(x,z) + \prod_{i = 1}^k b_i - \underline{\omega_n}(x,z).
		\end{equation}
		It remains to observe that inequality \eqref{eq:bs} holds for all directed $xz$-paths, so we can replace $\prod_{i = 1}^k b_i$ by $\underline{\omega_{n+1}}(x,z)$, obtaining
		$$
		\Pr(z \in A) \leq \Pr(x \in A) \cdot \underline{\omega_n}(x,z) + \underline{\omega_{n+1}}(x,z)-\underline{\omega_n}(x,z),
		$$
		as desired.
	\end{proof}
	
	\begin{proof}[Proof of Lemma~\ref{lemma:n}]
		Proof is by induction on $n$. For $n = 0$, the lemma simply asserts that $\Pr(y \in A) \leq 1$. Now assume that \eqref{eq:n} holds for some $n \in \mathbb{Z}_{\geq 0}$ and consider an edge $xy \in E^s$. Since $A$ is out-closed, $x \in A$ implies $y \in A$, so
		$$
			\Pr(y \in A) = \Pr(x \in A) + \Pr(x \not \in A \wedge y \in A).
		$$
		Since $F$ is an $A$-cut, it contains at least one edge $e \in E(x, y)$ whenever $x \not \in A$ and $y \in A$. Using the union bound, we obtain
		$$
		\Pr(x \not \in A \wedge y \in A) \leq \sum_{e \in E(x,y)} \Pr(e \in F \wedge y \in A).
		$$
		Thus,
		\begin{equation}\label{eq:yinA}
		\Pr(y \in A) \leq \Pr(x \in A) + \sum_{e \in E(x,y)} \Pr(e \in F \wedge y \in A).
		\end{equation}
		Let us now estimate $\Pr(e \in F \wedge y \in A)$ for each $e \in E(x,y)$. Consider any $z \in R_D(y)$. Since $A$ is out-closed,  $y \in A$ implies $z \in A$, so
		$$
		\Pr(e \in F \wedge y \in A) \leq \Pr(e \in F \wedge z \in A) = \Pr(e \in F \vert z \in A) \cdot \Pr(z \in A).
		$$
		Due to Claim~\ref{claim:generalizedn},
		$$
		\Pr\left(z \in A\right) \leq \Pr(x \in A) \cdot \underline{\omega_n}\left(x, z\right) + \underline{\omega_{n+1}}\left(x, z\right) - \underline{\omega_n}\left(x, z\right),
		$$
		so
		\begin{align*}
		\Pr(e \in F \wedge y \in A) &\leq \Pr\left(e \in F\middle\vert z \in A\right)\cdot \left(\Pr(x \in A) \cdot \underline{\omega_n}\left(x, z\right) + \underline{\omega_{n+1}}\left(x, z\right) - \underline{\omega_n}\left(x, z\right)\right) \\
		&= \Pr(x \in A) \cdot \rho_n(e, z) + \rho_{n+1}(e,z) - \rho_n(e, z).
		\end{align*}
		Since $\Pr(x \in A) \leq 1$ and
		$
		\rho_n(e,z) \geq \rho_n(e)
		$,
		we get
		$$
		\Pr(e \in F \wedge y \in A) \leq \Pr(x \in A) \cdot \rho_n(e) + \rho_{n+1}(e,z) - \rho_n(e).
		$$
		The last inequality holds for every $z\in R_D(y)$, so we can replace $\rho_{n+1}(e,z)$ in it by $\rho_{n+1}(e)$, obtaining
		\begin{equation}\label{eq:each_term}
		\Pr(e \in F \wedge y \in A) \leq \Pr(x \in A) \cdot \rho_n(e) + \rho_{n+1}(e) - \rho_n(e).
		\end{equation}
		Plugging~\eqref{eq:each_term} into~\eqref{eq:yinA}, we get
		\begin{align*}
		\Pr(y\in A) \,\leq\, \Pr(x\in A) 
		+ \sum_{e \in E(x,y)} \left(\Pr(x \in A) \cdot \rho_n(e) + \rho_{n+1}(e) - \rho_n(e)\right).
		\end{align*}
		The right hand side of the last inequality can be rewritten as
		\begin{align*}
		&\Pr(x \in A) \cdot \left(1 + \sum_{e \in E(x,y)} \rho_n(e)\right)
		+\, \sum_{e \in E(x,y)} \rho_{n+1}(e) - \sum_{e \in E(x,y)} \rho_{n}(e)\\
		=\,&\Pr(x \in A) \cdot f(\omega_{n})(xy) + f(\omega_{n+1})(xy) - f(\omega_{n})(xy)\\
		=\,& \Pr(x \in A) \cdot \omega_{n+1}(xy) + \omega_{n+2}(xy) - \omega_{n+1}(xy),
		\end{align*}
		as desired.
	\end{proof}
	
	\section{Applications}\label{sec:applications}
	
	\subsection{A special version of the~LCL}\label{subsec:special}
	
	In this subsection we introduce a particular and perhaps more intuitive set-up for the~LCL, that will be sufficient for almost all applications discussed in this paper.
	
	Let $I$ be a finite set. A family $A \in \powerset{\powerset{I}}$ of subsets of $I$ is \emph{downwards-closed} if for each $S \in A$, $\powerset{S} \subseteq A$. The \emph{boundary} $\partial A$ of a downwards-closed family is defined to be
	$$
	\partial A \coloneqq \{i \in I \,:\, S \in A \text{ and } S \cup \{i\} \not \in A \text{ for some } S \subseteq I \setminus \{i\}\}.
	$$
	
	Suppose that $\Omega$ is a probability space and $A \colon \Omega \to \powerset{\powerset{I}}$ is a random variable such that $A$ is downwards-closed with probability~$1$. Let $B$ be a random event and let $\tau \colon I \to [1;+\infty)$ be a function. For a subset $X \subseteq I$, let
	$$
		\tau(X) \coloneqq \prod_{i \in X} \tau(i),
	$$
	and
	\begin{equation}\label{eq:sigma_eq}
	\sigma^A_\tau(B, X) \coloneqq \max_{Z \subseteq I \setminus X} \Pr(B \text{ and } Z \cup X \not \in A \vert Z \in A) \cdot \tau(X).
	\end{equation}
	For most applications, the following upper bound is sufficient:
	\begin{equation}\label{eq:sigma_leq}
	\sigma^A_\tau(B, X) \leq \max_{Z \subseteq I \setminus X} \Pr(B\vert Z \in A) \cdot \tau(X).
	\end{equation}
	The only place in this paper where we use \eqref{eq:sigma_eq} directly instead of substituting the bound \eqref{eq:sigma_leq} is in the proof of Theorem~\ref{theo:acyclic}.
	Finally, for an element $i \in I$, let
	$$
	\sigma^A_\tau(B, i) \coloneqq \min_{i \in X \subseteq I} \sigma^A_\tau(B, X).
	$$
	The following statement is a straightforward, yet useful, corollary of the~LCL:
	
	\begin{theo}\label{theo:hypercubes}
		Let $I$ be a finite set. Let $\Omega$ be a probability space and let $A \colon \Omega \to \powerset{\powerset{I}}$ be a random variable such that with probability~$1$, $A$ is a nonempty downwards-closed family of subsets of $I$. For each $i \in I$, let $\mathcal{B}(i)$ be a finite collection of random events such that whenever $i \in \partial A$, at least one of the events in $\mathcal{B}(i)$ holds. Suppose that there is a function $\tau \colon I \to [1;+\infty)$ such that for all $i \in I$, we have
		\begin{equation}\label{eq:special_main}
		\tau(i) \geq 1 + \sum_{B \in \mathcal{B}(i)} \sigma^A_\tau(B, i).
		\end{equation}
		Then $\Pr(I \in A) \geq 1/\tau(I) > 0$.
	\end{theo}
	
	\begin{proof}
		For convenience, we may assume that for each $i \in I$, the set $\mathcal{B}(i)$ is nonempty (we can arrange that by adding the empty event to each $\mathcal{B}(i)$). Let $D$ be the digraph with vertex set $\powerset{I}$ and edge set $$E \coloneqq \{e_{i,S, B}\,:\, i \in I,\, S \subseteq I \setminus \{i\}, \, B \in \mathcal{B}(i)\},$$
		where the edge $e_{i, S, B}$ goes from $S \cup \{i\}$ to $S$. Thus, we have
		$$
			E^s = \{(S \cup \{i\}, S) \,:\, i \in I,\, S \subseteq I \setminus \{i\}\},
		$$
		which implies that for $S$, $Z \subseteq I$,
		$$
			Z \in R_D(S) \,\Longleftrightarrow\, Z \subseteq S.
		$$
		Moreover, if $Z \subseteq S \subseteq I$, then all directed $(S,Z)$-paths have length exactly $|S \setminus Z|$.
		
		Since $A$ is downwards-closed, it is out-closed in $D$. Let $F \colon \Omega \to \powerset{E}$ be a random set of edges defined by
		$$
			e_{i, S, B} \in F \,\Longleftrightarrow\, B \text{ holds and } S \cup \{i\} \not \in A.
		$$
		We claim that $F$ is an $A$-cut. Indeed, consider any edge $(S \cup \{i\}, S) \in E^s$ and suppose that we have $S \cup \{i\} \not \in A$ and $S \in A$. By definition, this means that $i \in \partial A$, so at least one event $B \in \mathcal{B}(i)$ holds. But then $e_{i, S, B} \in F \cap E(S \cup \{i\}, S)$, as desired.
		
		Let $\tau \colon I \to [1; +\infty)$ be a function satisfying~\eqref{eq:special_main} and let $\omega \colon E^s \to [1; +\infty)$ be given by $\omega((S \cup \{i\}, S)) \coloneqq \tau(i)$. Note that for any $Z \subseteq S \subseteq I$, we have $\underline{\omega}(S, Z) = \tau(S \setminus Z)$.
		
		\begin{small_claim}\label{claim:ineq}
			Let $i \in I$, $S \subseteq I \setminus \{i\}$, and $B \in \mathcal{B}(i)$. Then
			$$
				\rho^{A, F}_{\omega}(e_{i, S, B}) \leq \sigma^A_\tau(B, i).
			$$
		\end{small_claim}
		\begin{claimproof}
			Let $X$ be a set with $i \in X \subseteq I$ such that $\sigma^A_\tau(B, i) = \sigma^A_\tau(B, X)$ and let $Z \coloneqq S \setminus X$. We have
			\begin{align*}
			\rho^{A, F}_\omega(e_{i, S, B}) \leq \rho^{A,F}_\omega(e_{i, S, B}, Z)
				=\Pr(e_{i, S, B} \in F \vert Z \in A) \cdot \underline{\omega}(S \cup \{i\}, Z).
			\end{align*}
			Since $Z \cup X \supseteq S \cup \{i\}$ and $A$ is downwards-closed, we can write
		\begin{align*}
			\Pr(e_{i, S, B} \in F \vert Z \in A) \cdot \underline{\omega}(S \cup \{i\}, Z) \leq \Pr(B \text{ and } Z \cup X \not \in A \vert Z \in A) \cdot \tau((S \cup \{i\}) \setminus Z).
		\end{align*}
		Since $(S \cup \{i\}) \setminus Z \subseteq X$ and $\tau$ takes values in $[1; + \infty$), we have $\tau((S \cup \{i\}) \setminus Z) \leq \tau(X)$, so
		 $$
			\Pr(B \text{ and } Z \cup X \not \in A \vert Z \in A) \cdot \tau((S \cup \{i\}) \setminus Z)\leq \Pr(B \text{ and } Z \cup X \not \in A \vert Z \in A) \cdot \tau(X)
			\leq \sigma^A_\tau(B, X) = \sigma^A_\tau(B, i).\qedhere
		$$
		\end{claimproof}
		
		Let $(S \cup \{i\}, S) \in E^s$. Using~\eqref{eq:special_main} and Claim~\ref{claim:ineq}, we obtain
		$$
			\omega((S \cup \{i\}, S)) = \tau(i) \geq 1 + \sum_{B \in \mathcal{B}(i)} \sigma^A_\tau(B, i) \geq 1 + \sum_{B \in \mathcal{B}(i)} \rho^{A, F}_\omega(e_{i, S, B}) = 1 + \sum_{e \in E(S \cup \{i\}, S)} \rho^{A, F}_\omega(e),
		$$
		i.e., $\omega$ satisfies~\eqref{eq:main}. Thus, by Corollary~\ref{corl:positive},
		$$
			\Pr(I \in A) \geq \frac{\Pr(\emptyset \in A)}{\underline{\omega}(I, \emptyset)} = \frac{1}{\tau(I)} > 0,
		$$
		as desired. (Here we are using that $\Pr(\emptyset \in A) = 1$, which follows from the fact that with probability~$1$, $A$ is nonempty and downwards-closed.)
	\end{proof}
	
	\subsection{The LCL implies the Lopsided LLL}\label{subsec:LLL}
	
	In this subsection we use the LCL to prove the Lopsided LLL, which is a strengthening of the standard~LLL.
	
	\begin{theo}[Lopsided Lov\'{a}sz Local Lemma, \cite{Erdos}]
		Let $B_1$, \ldots, $B_n$ be random events in a probability space $\Omega$. For each $1 \leq i \leq n$, let $\Gamma(i)$ be a subset of $\{1, \ldots, n\} \setminus \{i\}$ such that for all $Z \subseteq \{1, \ldots, n\} \setminus (\Gamma(i) \cup \{i\})$, we have
		\begin{equation}\label{eq:LopLLLcond}
		\Pr\left(B_i \middle\vert \bigcap_{j \in Z} \overline{B_j} \right) \leq \Pr(B_i).
		\end{equation}
		Suppose that there exists a function $\mu\colon \{1, \ldots, n\}\to[0;1)$ such that for every $1 \leq i \leq n$, we have
		\begin{equation}\label{eq:LopLLL}
		\Pr(B_i)\leq \mu(i) \prod_{j \in \Gamma(i)}(1-\mu(j)).
		\end{equation}
		Then
		$$
		\Pr\left(\bigcap_{i = 1}^n\overline{B_i}\right) \geq \prod_{i  = 1}^n(1 - \mu(i)) > 0.
		$$
	\end{theo}
	
	\begin{proof}
		We will use Theorem~\ref{theo:hypercubes}. Set $I \coloneqq \{1, \ldots, n\}$ and let $I_0 \colon \Omega \to \powerset{I}$ and $I_1 \colon \Omega \to \powerset{I}$ be random variables defined by
		$$
		I_1 \coloneqq \{i \in I\,:\, B_i \text{ holds}\} \;\;\;\text{ and }\;\;\; I_0 \coloneqq I \setminus I_1.
		$$
		Set $A \coloneqq \powerset{I_0}$. In other words, a set $S\subseteq I$ belongs $A$ if and only if $\bigcap_{i \in S} \overline{B_i}$ holds.  It follows that $A$ is a nonempty downwards-closed family of subsets of $I$ and $\partial A = I_1$ (i.e., $i \in \partial A$ if and only if $B_i$ holds). Therefore, we can apply Theorem~\ref{theo:hypercubes} with $\mathcal{B}(i) \coloneqq \{B_i\}$ for each $i \in I$.
		
		By~\eqref{eq:LopLLLcond}, if $i \in I$ and $Z \subseteq I \setminus (\Gamma(i)\cup \{i\})$, then
		$$
		\Pr(B_i \vert Z \in A) = \Pr\left(B_i\middle\vert \bigcap_{j \in Z} \overline{B_j}\right) \leq \Pr(B_i).
		$$
		Thus, for any $i \in I$ and $\tau \colon I \to [1;+\infty)$, we have
		$$
		\sigma^A_\tau(B_i, i) \leq \sigma^A_\tau(B_i, \Gamma(i) \cup \{i\}) \leq \max_{Z \subseteq I \setminus (\Gamma(i) \cup \{i\})} \Pr(B_i \vert Z \in A) \cdot \tau(\Gamma(i) \cup \{i\}) \leq \Pr(B_i) \cdot \tau(\Gamma(i) \cup \{i\}).
		$$
		Therefore, \eqref{eq:special_main} holds as long as for each $i \in I$, we have
		\begin{equation}\label{eq:LopLLL_for_tau}
		\tau(i) \geq 1 + \Pr(B_i) \cdot \tau(\Gamma(i) \cup \{i\}).
		\end{equation}
		Suppose that $\mu \colon I \to [0;1)$ satisfies~\eqref{eq:LopLLL}. We claim that $\tau(i) \coloneqq 1/(1-\mu(i))$ satisfies~\eqref{eq:LopLLL_for_tau}. Indeed,
		\begin{align*}
		1 + \Pr(B_i) \cdot \tau(\Gamma(i) \cup \{i\}) &= 1 + \Pr(B_i) \cdot \prod_{j \in \Gamma(i) \cup \{i\}} \tau(j)\\
		&= 1 + \frac{\Pr(B_i)}{\prod_{j \in \Gamma(i)\cup\{i\}} (1-\mu(j))} \\
		[\text{by~\eqref{eq:LopLLL}}]\;\;\;& \leq 1 + \frac{\mu(B_i) \prod_{j\in \Gamma(i)}(1-\mu(j))}{\prod_{j \in \Gamma(i)\cup\{i\}} (1-\mu(j))} \\
		& = 1 + \frac{\mu(i)}{1-\mu(i)} = \frac{1}{1-\mu(i)} = \tau(i),
		\end{align*}
		Theorem~\ref{theo:hypercubes} now yields
		$$
		\Pr\left(\bigcap_{i = 1}^n\overline{B_i}\right) = \Pr(I \in A) \geq \frac{1}{\tau(I)} = \frac{1}{\prod_{i = 1}^n\tau(i)}= \prod_{i = 1}^n(1 - \mu(i)),
		$$
		as desired.
	\end{proof}
	
	\begin{remk}
		The above derivation of the Lopsided LLL from Theorem~\ref{theo:hypercubes} clarifies the precise relationship between the two statements. Essentially, Theorem~\ref{theo:hypercubes} reduces to the classical LLL under the following two main assumptions: (1) the set $A$ contains an inclusion-maximum element; and (2) each of the sets $\mathcal{B}(i)$ is a singleton, containing only one ``bad'' event. Neither of these assumptions is satisfied in the applications discussed later, where the~LCL outperforms the~LLL.
	\end{remk}
	
	\subsection{First example: hypergraph coloring}\label{subsec:hypcol}
	
	In this subsection we provide some intuition behind the~LCL using a very basic example: coloring uniform hypergraphs with $2$ colors.
	
	Let $\mathcal{H}$ be a $d$-regular $k$-uniform hypergraph with vertex set $V$ and edge set $E$, and suppose we want to establish a relation between $d$ and $k$ that guarantees that $\mathcal{H}$ is $2$-colorable. A straightforward application of the LLL gives the bound
	$$
	\frac{e}{2^{k-1}}((d-1)k+1) \leq 1,
	$$
	which is equivalent to
	\begin{equation}\label{eq:hypcolLLL}
	d \leq \frac{2^{k-1}}{ek} +1 - \frac{1}{k}.
	\end{equation}
	
	Let us now explain how to apply the~LCL (in the simplified form of Theorem~\ref{theo:hypercubes}) to this problem. Choose a coloring $\varphi \colon V \to \{\text{red}, \text{blue}\}$ uniformly at random. Define $A \subseteq \powerset{V}$ by
	$$
	A \coloneqq \{ S \subseteq V\,:\, \text{there is no $\varphi$-monochromatic edge $H \subseteq S$}\}.
	$$
	Clearly, $A$ is downwards-closed, and, since we always have $\emptyset \in A$, $A$ is nonempty. Moreover, $V \in A$ if and only if $\varphi$ is a proper coloring of $\mathcal{H}$. Therefore, if we can apply Theorem~\ref{theo:hypercubes} to show that $\Pr(V \in A) > 0$, then $\mathcal{H}$ is $2$-colorable.
	
	In order to apply Theorem~\ref{theo:hypercubes}, we have to specify, for each $v \in V$, a finite family $\mathcal{B}(v)$ of ``bad'' random events such that whenever $v \in \partial A$, at least one of the events in $\mathcal{B}(v)$ holds. Notice that if $v \in \partial A$, i.e., for some $S \subseteq V \setminus \{v\}$, we have $S \in A$ and $S \cup \{v\} \not \in A$, then there must exist at least one $\varphi$-monochromatic edge $H \ni v$. Thus, we can set
	$$
		\mathcal{B}(v) \coloneqq \{B_H \,:\, v \in H \in E\},
	$$
	where the event $B_H$ happens is and only if $H$ is $\varphi$-monochromatic. Since $\mathcal{H}$ is $d$-regular, $|\mathcal{B}(v)| = d$.
	
	We will assume that $\tau (v) = \tau \in [1; + \infty)$ is a constant function. In that case, for any $S \subseteq V$, $\tau(S) = \tau^{|S|}$. Let $v \in V$ and let $H \in E$ be such that $H \ni v$. To verify~\eqref{eq:special_main}, we require an upper bound on the quantity $\sigma^A_\tau(B_H, v)$. By definition, $$\sigma^A_\tau(B_H, v) = \min_{v \in X \subseteq V} \sigma^A_\tau(B_H, X),$$ so it is sufficient to upper bound $\sigma^A_\tau(B_H, X)$ for some set $X \ni v$. Since $$\sigma^A_\tau(B_H, X) = \max_{Z \subseteq V \setminus X} \Pr(B_H \vert Z \in A) \cdot \tau^{|X|},$$
	we just need to find a set $X \ni v$ such that the conditional probability $\Pr(B_H \vert Z \in A)$ for $Z \subseteq V \setminus X$ is easy to bound. Moreover, we would like $|X|$ to be as small as possible (to minimize the factor $\tau^{|X|}$).
	
	Since the colors of distinct vertices are independent, the events $B_H$ and ``$Z \in A$'' are independent whenever $Z \cap H = \emptyset$. Therefore, for $Z \subseteq V \setminus H$,
	\begin{equation}\label{eq:prob}
		\Pr(B_H \vert Z \in A) \leq \Pr(B_H) = \frac{1}{2^{k-1}}.
	\end{equation}
	(The inequality might be strict if $\Pr(Z \in A) = 0$, in which case $\Pr(B_H\vert Z \in A) = 0$ as well, due to our convention regarding conditional probabilities; see Remark~\ref{remk:cond_prob}.) Thus, it is natural to take $X = H$, which gives
	$$
		\sigma^A_\tau(B_H, v) \leq \sigma^A_\tau(B_H, H) \leq \max_{Z \subseteq V \setminus H} \Pr(B_H \vert Z \in A) \cdot \tau^{|H|} \leq \frac{\tau^k}{2^{k-1}}.
	$$
	Hence it is enough to ensure that $\tau$ satisfies
	$$
	\tau \geq  1 + \frac{d\tau^k}{2^{k-1}}.
	$$
	A straightforward calculation shows that the following condition is sufficient:
	\begin{equation}\label{eq:hypcol}
	d \leq \frac{2^{k-1}}{k} \left(1-\frac{1}{k}\right)^{k-1},
	\end{equation}
	or, a bit more crudely,
	\begin{equation}\label{eq:hypcol1}
	d \leq \frac{2^{k-1}}{ek},
	\end{equation}
	which is almost identical to \eqref{eq:hypcolLLL}. Note that the precise bound \eqref{eq:hypcol} is, in fact, better than \eqref{eq:hypcolLLL} for $k \geq 10$.
	
	We can improve \eqref{eq:hypcol1} slightly by estimating $\sigma^A_\tau(B_H, v)$ more carefully. Observe that the inequality~\eqref{eq:prob} holds even if $|Z \cap H| = 1$ (because fixing the color of one of the vertices in $H$ does not change the probability that $H$ is monochromatic). Therefore, upon choosing any vertex $u \in H \setminus \{v\}$ and taking $X = H \setminus \{u\}$, we obtain
	$$
		\sigma^A_\tau(B_H, v) \leq \sigma^A_\tau(B_H, H \setminus \{u\}) \leq \max_{Z \subseteq (V \setminus H) \cup \{u\}} \Pr(B_H \vert Z \in A) \cdot \tau^{|H \setminus \{u\}|} \leq \frac{\tau^{k-1}}{2^{k-1}}.
	$$
	Thus, it is enough to ensure that
	$$
	\tau \geq 1 + \frac{d\tau^{k-1}}{2^{k-1}},
	$$
	which can be satisfied as long as
	\begin{equation}\label{eq:hypcol2}
	d \leq \frac{2^{k-1}}{e(k-1)}.
	\end{equation}
	
	The bound \eqref{eq:hypcol2} is better than \eqref{eq:hypcol1} by a quantity of order $\Omega\left(2^k\middle/k^2\right)$. This is, of course, not a significant improvement (and the bound is still considerably weaker than the best known result due to Radhakrishnan and Srinivasan~\cite{RS}, namely $d \leq \epsilon2^k /\sqrt{k\log k}$ for some absolute constant $\epsilon > 0$). However, the observation that helped us improve~\eqref{eq:hypcol1} to~\eqref{eq:hypcol2} highlights one of the important strengths of the~LCL. The fact that $\Pr(B_H \vert Z \in A) \leq 1/2^{k-1}$ for all $Z$ such that $|Z \cap H| \leq 1$ (and not only when $Z \cap H = \emptyset$) contains information beyond the individual probabilities of ``bad'' events and their dependencies, and the~LCL has a mechanism for putting that additional information to use. Similar ideas will reappear several times in later applications.	
	
	\subsection{Nonrepetitive sequences and nonrepetitive colorings}\label{subsec:nonrep}
	
	A finite sequence $a_1a_2 \ldots a_n$ is \emph{nonrepetitive} if it does not contain the same nonempty substring twice in a row, i.e., if there are no $s$, $1 \leq s \leq n-1$, and $t$, $1 \leq t \leq \left\lfloor (n-s+1)/2\right\rfloor$, such that $a_k = a_{k+t}$ for all $s \leq k \leq s+t-1$. A well-known result by Thue \cite{Thue} asserts that there exist arbitrarily long nonrepetitive sequences of elements from $\{0, 1, 2\}$. The next theorem is a choosability version of Thue's result. It was the first example of a new combinatorial bound obtained using the entropy compression method that surpasses the analogous bound provided by a direct application of the LLL.
	
	\begin{theo}[Grytczuk--Przyby\l{}o--Zhu \cite{GPZ}; Grytczuk--Kozik--Micek \cite{Grytczuk}]\label{theo:nonrepseq}
		Let $L_1$, $L_2$, \ldots, $L_n$ be a sequence of sets with $|L_i| \geq 4$ for all $1\leq i \leq n$. Then there exists a nonrepetitive sequence $a_1 a_2 \ldots a_n$ such that $a_i \in L_i$ for all $1 \leq i \leq n$.
	\end{theo}
	
	Note that it is an open problem whether the same result holds for $|L_i| \geq 3$.
	
	\begin{proof}
		This is the only example in this paper where the~LCL is applied directly, without reducing it to Theorem~\ref{theo:hypercubes}. Let $P$ be the directed path of length $n$ with vertex set $V \coloneqq \{v_1, \ldots, v_n\}$ and with edges of the form $(v_{i+1}, v_i)$ for all $1 \leq i \leq n-1$. Choose a random sequence $a_1a_2\ldots a_n$ by selecting each $a_i \in L_i$ uniformly and independently from each other.
		Define a set $A \subseteq V$ as follows:
		$$
			v_i \in A \,\Longleftrightarrow\, \text{$a_1a_2\ldots a_i$ is a nonrepetitive sequence}.
		$$
		Note that $A$ is out-closed, $\Pr(v_1 \in A) = 1$, and $v_n \in A$ if and only if $a_1a_2\ldots a_n$ is a nonrepetitive sequence.
		
		Consider an edge $(v_{i+1}, v_i)$ of $P$. If $v_i \in A$ but $v_{i+1} \not \in A$, then there exist $s$ and $t$ such that
		$$
			s+2t-1 = i+1
		$$
		and $a_k = a_{k+t}$ for all $s \leq k \leq s+t-1$ (i.e., $a_s a_{s+1} \ldots a_{i+1}$ is a repetition). This observation motivates the following construction. Let $D$ be the digraph such that $D^s = P$ and for each $(v_{i+1}, v_i) \in E(P)$ and $s$, $t$ with $s+2t-1=i+1$, there is a corresponding edge $e_{s,t} \in E(D)$ going from $v_{i+1}$ to $v_i$. Let
		$$
			e_{s,t} \in F\,\Longleftrightarrow\, \text{$a_k = a_{k+t}$ for all $s \leq k \leq s+t-1$}.
		$$
		Then $F$ is an $A$-cut (see Fig.~\ref{fig:path}). Note that for each fixed $t \geq 1$, there exists at most one $s$ such that $s+2t-1 = i+1$, so there is at most one edge of the form $e_{s,t} \in E(v_{i+1}, v_i)$, where $E$ denotes the edge set of $D$.
		
		\begin{figure}[h]
			\begin{mdframed}[linewidth=0.5pt]
			\centering
			\includegraphics{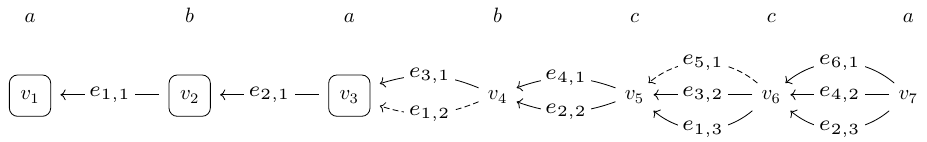}
			\caption{\label{fig:path} For $n = 7$ and a sequence $a_1 a_2 a_3 a_4 a_5 a_6 a_7 = ababcca$, we have $A = \{v_1, v_2, v_3\}$ (since the first $4$~letters contain a repetition) and $F = \{e_{1, 2}, e_{5,1}\}$ (due to the repetitions $\pmb{abab}cca$ and $abab\pmb{cc}a$).
				}
			\end{mdframed}
		\end{figure}
		
		A vertex $v_j$ is reachable from $v_i$ if and only if $j \leq i$. In particular, if $s+2t-1 = i+1$, then $v_{s+t-1}$ is reachable from $v_i$. Observe that the probability of $a_k = a_{k+t}$ is at most $1/|L_{k+t}|$, even if the value of $a_k$ is fixed. Therefore, for $e_{s, t} \in E(v_{i+1}, v_i)$, we have
		\begin{align*}
			\Pr\left(e_{s,t} \in F\middle\vert v_{s+t-1} \in A\right) \,&=\, \Pr\left(a_k = a_{k+t} \text{ for all }s\leq k \leq s+t-1 \middle\vert v_{s+t-1} \in A\right) \\
			&\leq\, \prod_{k=s}^{s+t-1} \frac{1}{|L_{k+t}|} \leq \frac{1}{4^t}.
		\end{align*}
		If $\omega(v_{i+1},v_i) = \omega \in [1; +\infty)$ is a fixed constant, then for all $i \geq j$, $\underline{\omega}(v_i, v_j) = \omega^{i-j}$.
		In particular, if $s+2t-1 = i+1$, then
		$$
			\underline{\omega}(v_{i+1}, v_{s+t-1}) = \omega^t.
		$$
		Thus,
		$$
			\rho^{A,F}_\omega(e_{s,t}) \leq \rho^{A,F}_\omega(e_{s,t}, v_{s+t-1}) = \Pr\left(e_{s,t} \in F\middle\vert v_{s+t-1} \in A\right) \cdot \underline{\omega}(v_{i+1}, v_{s+t-1})  \leq \frac{\omega^t}{4^t}.
		$$
		Hence, it is enough to find a constant $\omega \in [1; + \infty)$ such that
		$$
			\omega \geq 1 + \sum_{t=1}^\infty \frac{\omega^t}{4^t} = \frac{1}{1 -\omega/4},
		$$
		where the last equality is subject to $\omega < 4$. Setting $\omega = 2$ completes the proof. 
	\end{proof}
	
	A vertex coloring $\varphi$ of a graph $G$ is \emph{nonrepetitive} if there is no path $P$ in $G$ with an even number of vertices such that the first half of $P$ receives the same sequence of colors as the second half of $P$, i.e., if there is no path $v_1$, $v_2$, \ldots, $v_{2t}$ of length $2t$ such that $\varphi(v_k) = \varphi(v_{k+t})$ for all $1\leq k \leq t$. The least number of colors that is needed for a nonrepetitive coloring of $G$ is called the \emph{nonrepetitive chromatic number} of $G$ and is denoted by $\pi(G)$.
		
	The first upper bound on $\pi(G)$ in terms of the maximum degree $\Delta(G)$ was given by Alon, Grytczuk, Ha\l{}uszczak, and Riordan \cite{Alon3}, who proved that there is a constant $c$ such that $\pi(G)\leq c \Delta(G)^2$. Originally this result was obtained with $c = 2e^{16}$. The constant was then improved to $c = 16$ by Grytczuk \cite{Grytczuk1}, and then to $c = 12.92$ by Harant and Jendrol' \cite{Haranta}. All these results were based on the LLL.
		
	Dujmovi\'{c}, Joret, Kozik, and Wood \cite{Duj} managed to decrease the value of the aforementioned constant $c$ dramatically using the entropy compression method. Namely, they lowered the constant to $1$, or, to be precise, they showed that $\pi(G)\leq(1+o(1))\Delta(G)^2$ (assuming $\Delta(G)\to\infty$).
		
	The currently best known bound is given by the following theorem.
		
	\begin{theo}[Gon\c{c}alves--Montassier--Pinlou \cite{Goncalves}]
		For every graph $G$ with maximum degree $\Delta$,
			$$\pi(G) \leq \left\lceil \Delta^2 + \frac{3}{2^{2/3}}\Delta^{5/3} +\frac{2^{2/3} \Delta^{5/3}}{\Delta^{1/3}-2^{1/3}} \right\rceil.$$
	\end{theo}
	
	\begin{proof}
		Suppose that
		\begin{equation}\label{eq:non-repbound}
			k \geq \Delta^2 + \frac{3}{2^{2/3}}\Delta^{5/3} +\frac{2^{2/3} \Delta^{5/3}}{\Delta^{1/3}-2^{1/3}}.
		\end{equation}
		We will use Theorem~\ref{theo:hypercubes} to show that $G$ has a nonrepetitive $k$-coloring.
		
		For brevity, let $V \coloneqq V(G)$ and $E \coloneqq E(G)$. Choose a $k$-coloring $\varphi$ of $G$ uniformly at random. Define a set $A \subseteq \powerset{V}$ by
		$$
			A \coloneqq \{S \subseteq V\,:\, \text{$\varphi$ is a nonrepetitive coloring of $G[S]$}\},
		$$
		where $G[S]$ denotes the induced subgraph of $G$ with vertex set $S$. Note that $A$ is downwards-closed and nonempty with probability $1$, and $V \in A$ if and only if $\varphi$ is a nonrepetitive coloring of $G$.
		
		Consider any $v \in V$. If $v \in \partial A$, then there exists a path $P \ni v$ of even length that is colored repetitively by $\varphi$. Thus, we can set
		$$
			\mathcal{B}(v) \coloneqq \{B_P \,:\, \text{$P \ni v$ is a path of even length}\},
		$$
		where the event $B_P$ happens if and only if $P$ is colored repetitively by $\varphi$.
		
		The number of events in $\mathcal{B}(v)$ corresponding to paths of some fixed length $2t$ is equal to the number of all paths~$P$ of length $2t$ passing through $v$, which does not exceed $t \Delta^{2t -1}$. Indeed, if $P = v_1$, $v_2$, \ldots, $v_{2t}$, then we can assume $v$ is one of the vertices $v_1$, $v_2$, \ldots, $v_t$, so there are $t$ ways to choose the position of $v$ on $P$. After the position of $v$ has been determined, we can select all other vertices one by one so that each time we are choosing only from the neighbors of one of the previous vertices. Since the maximum degree of $G$ is $\Delta$, we get the bound~$t \Delta^{2t-1}$, as desired.
		
		We will assume $\tau(v) = \tau \in [1;+\infty)$ is a constant. We need to upper bound $\sigma^A_\tau(B_P, v)$ for each $v \in V$ and a path $P \ni v$ of length $2t$. Let $P'$ be the half of $P$ that contains $v$. Note that if $Z \subseteq V \setminus P'$, then $ \Pr(B_P \vert Z \in A) \leq  1/k^t$, since the coloring of $P'$ is independent from the coloring of $Z$. Therefore,
		$$
			\sigma^A_\tau(B_P, v) \leq \sigma^A_\tau(B_P, P') \leq \max_{Z \subseteq V \setminus P'} \Pr(B_P \vert Z \in A) \cdot \tau^{|P'|} \leq \frac{\tau^t}{k^t}.
		$$
		Hence, it is enough to ensure that there exists $\tau \in [1;+\infty)$ such that
		\begin{equation}\label{eq:non-rep}
			\tau \geq 1 + \sum_{t = 1}^\infty t\Delta^{2t-1} \cdot \frac{\tau^t}{k^t} = 1 + \frac{\Delta \tau/k}{(1-\Delta^2 \tau/k)^2},
		\end{equation}
		where the last equality is subject to $\Delta^2 \tau/k < 1$. Setting $y \coloneqq \Delta^2 \tau/k$, we can rewrite \eqref{eq:non-rep} as
		\begin{equation}\label{eq:non-rep2}
			\frac{k}{\Delta^2} \geq \frac{1}{y} + \frac{1}{\Delta(1-y)^2}.
		\end{equation}
		Following Gon\c{c}alves et al., we take $y = 1 - \left(2/\Delta\right)^{1/3}$, and \eqref{eq:non-rep2} becomes
		$$
			\frac{k}{\Delta^2} \geq 1+ \frac{3}{2^{2/3}\Delta^{1/3}} +\frac{2^{2/3}}{\Delta^{2/3}-(2\Delta)^{1/3}}, 
		$$
		which is true by \eqref{eq:non-repbound}.
	\end{proof}
	
	\subsection{Acyclic edge colorings}\label{subsec:acyclic}
	
	An edge coloring of a graph $G$ is called an \emph{acyclic edge coloring} if it is proper (i.e. adjacent edges receive different colors) and every cycle in $G$ contains edges of at least three different colors (there are no \emph{bichromatic cycles} in $G$). The least number of colors needed for an acyclic edge coloring of $G$ is called the \emph{acyclic chromatic index} of $G$ and is denoted by $a'(G)$. The notion of acyclic (vertex) coloring was first introduced by Gr\"{u}nbaum \cite{Grunbaum}. The edge version was first considered by Fiam\v{c}ik \cite{Fiamcik}, and independently by Alon, McDiarmid, and Reed \cite{Alon1}.
	
	As in the case of nonrepetitive colorings, it is quite natural to ask for an upper bound on the acyclic chromatic index of a graph $G$ in terms of its maximum degree $\Delta(G)$. Since $a'(G)\geq \chi'(G) \geq \Delta(G)$, where $\chi'(G)$ denotes the ordinary chromatic index of $G$, this bound must be at least linear in $\Delta(G)$. The first linear bound was given by Alon et al. \cite{Alon1}, who showed that $a'(G)\leq 64 \Delta(G)$. Although it resolved the problem of determining the order of growth of $a'(G)$ in terms of $\Delta(G)$, it was conjectured that the sharp bound should be lower.
	
	\begin{conj}[Fiam\v{c}ik \cite{Fiamcik}; Alon--Sudakov--Zaks \cite{Alon2}]\label{conj:AECC}
		For every graph $G$, $a'(G) \leq \Delta(G)+2$.
	\end{conj}
	
	Note that the bound in Conjecture \ref{conj:AECC} is only one more than Vizing's bound on the chromatic index of $G$. However, this elegant conjecture is still far from being proven.
	
	The first major improvement to the bound $a'(G)\leq 64 \Delta(G)$ was made by Molloy and Reed \cite{Molloy}, who proved that $a'(G) \leq 16 \Delta(G)$. This bound remained the best for a while, until Ndreca, Procacci, and Scoppola \cite{Ndreca} managed to improve it to $a'(G)\leq \left\lceil 9.62(\Delta(G)-1)\right\rceil$. Again, first bounds for $a'(G)$ were obtained using the LLL. The bound $a'(G)\leq \left\lceil 9.62(\Delta(G)-1)\right\rceil$ by Ndreca et al. used an improved version of the LLL due to Bissacot, Fern\'{a}ndez, Procacci, and Scoppola \cite{Bissacot}.
	
	The best current bound for $a'(G)$ in terms of $\Delta(G)$ was obtained by Esperet and Parreau via the entropy compression method.
	\begin{theo}[Esperet--Parreau \cite{Esperet}]\label{theo:acyclic}
		For every graph $G$ with maximum degree $\Delta$, $a'(G) \leq 4(\Delta-1)$.
	\end{theo}
	
	\begin{proof}		
		We will apply Theorem~\ref{theo:hypercubes}. In this application, it will be important to use \eqref{eq:sigma_eq} instead of \eqref{eq:sigma_leq}. For brevity, let $V \coloneqq V(G)$ and $E \coloneqq E(G)$. Choose a $4(\Delta - 1)$-edge coloring $\varphi$ of $G$ uniformly at random. Call a cycle $C$ of length $2t$ \emph{$\varphi$-bichromatic} if $C= e_1$, $e_2$, \ldots, $e_{2t}$ and $\varphi(e_{2i-1}) = \varphi(e_{2t-1})$, $\varphi(e_{2i}) = \varphi(e_{2t})$ for all $1 \leq i \leq t-1$.
		
		Let
		$$
			A \coloneqq \{ S \subseteq E \,:\, \text{$\varphi$ is an acyclic edge coloring of $G[S]$}\},
		$$
		where $G[S]$ is the graph obtained from $G$ by removing all the edges outside $S$. Note that with probability $1$, $A$ is a nonempty downwards-closed family of subsets of $E$, and $E \in A$ if and only if $\varphi$ is an acyclic edge coloring of $G$.
		
		Consider any $e \in E$. If $e \in \partial A$, then either there exists an edge $e'$ adjacent to $e$ such that $\varphi(e) = \varphi(e')$, or there exists a $\varphi$-bichromatic cycle $C \ni e$ of even length. The crucial idea of \cite{Esperet} (which is credited to Jakub Kozik by the authors) is to handle $4$-cycles and cycles of length at least $6$ separately. Set
		$$
			\mathcal{B}(e) \coloneqq \{B_C \,:\, \text{$C \ni e$ is a cycle of length $2t \geq 6$}\} \cup \{B_e\},
		$$
		where
		\begin{enumerate}
			\item $B_C$ happens if and only if the cycle $C$ is $\varphi$-bichromatic;
			\item $B_e = \bigcap_C \overline{B_C}$, where the intersection is taken over all cycles $C \ni e$ of even length at least $6$. 
		\end{enumerate}
		
		Again, we will assume that $\tau(e) = \tau \in [1;+\infty)$ is a constant. Consider the event $B_e \in \mathcal{B}(e)$ of the second kind. It definitely does not look like a typical ``bad'' event. Recall, however, that in order to apply Theorem~\ref{theo:hypercubes}, we actually do not have to bound the conditional probability $\Pr(B_e \vert Z \in A)$; instead, we only need to work with the somewhat more complicated expression $\Pr(B_e \text{ and } Z \cup X \not \in A \vert Z \in A)$. To that end, we will use the following claim, which also plays a crucial role in the original proof by Esperet and Parreau.
		
		\begin{small_claim}\label{claim:acyclicext}
			Suppose that some edges of $G$ are properly colored. If $e \in E$ is uncolored, then there exist at most $2(\Delta-1)$ ways to color $e$ so that the resulting coloring either is not proper, or contains a bichromatic $4$-cycle going through $e$.
		\end{small_claim}
		\begin{claimproof}	
			Indeed, denote the given proper partial coloring by $\psi$ and let $e = uv$. Let $L_1$ (resp. $L_2$) be the set of colors appearing on the edges incident to $u$ (resp. $v$). The coloring becomes not proper if $e$ is colored using a color from $L_1 \cup L_2$, so there are $|L_1 \cup L_2 |$ such options. Suppose that coloring $e$ with color $c$ creates a bichromatic $4$-cycle $uvxy$. Then $c = \psi(xy)$ and $\psi(vx) = \psi(uy)$. Hence, the number of such colors $c$ is at most the number of pairs of edges $vx$, $uy$ such that $\psi(vx) = \psi(uy)$. Note that, since $\psi$ is proper, there can be at most one pair $vx$, $uy$ such that $\psi(vx) = \psi(uy) = c'$ for a particular color $c'$. Therefore, the total number of such pairs is exactly $|L_1 \cap L_2|$. Thus, there are at most $|L_1 \cup L_2| + |L_1 \cap L_2| = |L_1| + |L_2| \leq 2(\Delta - 1)$ ``forbidden'' colors for $e$, as desired.
		\end{claimproof}
		
		Let $Z \subseteq E \setminus \{e\}$. If $Z \in A$ while $Z \cup \{e\} \not \in A$, then either there is an edge $e'$ adjacent to $e$ such that $\varphi(e) = \varphi(e')$, or there exists a $\varphi$-bichromatic cycle $C \ni e$ of even length. If we additionally assume that $B_e$ holds, then the cycle $C$ must be of length $4$. Hence, we can use Claim~\ref{claim:acyclicext} to obtain
		$$
			\Pr(B_e \text{ and } Z  \cup \{e\} \not \in A \vert Z \in A) \leq \frac{2(\Delta - 1)}{4(\Delta - 1)} = \frac{1}{2}
		$$
		for all $Z \subseteq E \setminus \{e\}$. Therefore,
		$$
			\sigma^A_\tau(B_e, e) \leq \sigma^A_\tau(B_e, \{e\}) = \max_{Z \subseteq E \setminus \{e\}} \Pr(B_e \text{ and } Z  \cup \{e\} \not \in A\vert Z \in A) \cdot \tau^{|\{e\}|} \leq  \frac{\tau}{2}.
		$$
		
		Now we deal with the events of the form $B_C \in \mathcal{B}(e)$. Note that there are at most $(\Delta - 1)^{2t - 2}$ cycles of length $2t$ passing through $e$. Therefore, the number of events in $\mathcal{B}(e)$ corresponding to cycles of length $2t$ is at most $(\Delta-1)^{2t-2}$. Consider any such event $B_C$. Suppose that $C = e_1$, $e_2$, \ldots, $e_{2t}$, where $e_1 = e$. Then $B_C$  happens if and only if $\varphi(e_{2i-1}) = \varphi(e_{2t-1})$ and $\varphi(e_{2i}) = \varphi(e_{2t})$ for all $1 \leq i \leq t-1$. Even if the colors of $e_{2t-1}$ and $e_{2t}$ are fixed, the probability of this happening is $1/(4(\Delta - 1))^{2t-2}$. Due to this observation, if $C' \coloneqq \{e_1, e_2, \ldots, e_{2t-2}\}$ and $Z \subseteq E \setminus C'$, then $\Pr(B_C\vert Z \in A) \leq 1/(4(\Delta - 1))^{2t-2}$. Therefore,
		$$
			\sigma^A_\tau(B_C, e) \leq \sigma^A_\tau(B_C, C') \leq \max_{Z \subseteq E \setminus C'} \Pr(B_C \vert Z \in A) \cdot \tau^{|C'|}\leq \frac{\tau^{2t-2}}{(4(\Delta - 1))^{2t-2}}.
		$$
		 
		Putting everything together, it is enough to find a constant $\tau \in [1; +\infty)$ such that
		$$
			\tau \geq 1 + \sum_{t = 3}^\infty (\Delta - 1)^{2t-2} \cdot \frac{\tau^{2t-2}}{(4(\Delta - 1))^{2t-2}} + \frac{\tau}{2} = 1 + \frac{\left(\tau/4\right)^4}{1-\left(\tau/4\right)^2}+\frac{\tau}{2},
		$$
		where the last equality is valid if $\tau/4 < 1$. Setting $\tau = 2(\sqrt{5}-1)$ completes the proof.
	\end{proof}
	
	Further applications of the~LCL to acyclic edge coloring can be found in~\cite{Bernshteyn}.
	
	\subsection{Color-critical hypergraphs}\label{subsec:critical}	
	
	 A hypergraph $\mathcal{H}$ is \emph{$(k+1)$-critical} if it is not $k$-colorable, but each of its proper subhypergraphs is. Call a hypergraph $\mathcal{H}$ \emph{true} if all its edges have size at least $3$. It is interesting to know what the least possible number of edges in a $(k+1)$-critical true hypergraph on $n$ vertices is. The best known constructions due to Abbott and Hare~\cite{Abbott1} and Abbott, Hare, and Zhou~\cite{Abbott2} contain roughly $(k-1)n$ edges. This bound is asymptotically tight for $k \to \infty$, as the following theorem due to Kostochka and Stiebitz asserts.
	 
	 \begin{theo}[Kostochka--Stiebitz~\cite{Kostochka}]\label{theo:Kostochka}
		 	Every $(k+1)$-critical true hypergraph with $n$ vertices contains at least $(k-3k^{2/3})n$ edges.
	 \end{theo}
	 
	 Here we improve this result, obtaining the following new bound.
	 
	 \begin{theo}
	 	Every $(k+1)$-critical true hypergraph with $n$ vertices contains at least $(k-4\sqrt{k})n$ edges.
	 \end{theo}
	 \begin{proof}
	 	Our proof is essentially the same as the proof of Theorem~\ref{theo:Kostochka} given in \cite{Kostochka}. The only difference is that we replace the application of the LLL by an application of the LCL (in the form of Theorem~\ref{theo:hypercubes}).
	 	
		Let $\mathcal{H}$ be a $(k+1)$-critical true hypergraph with $n$ vertices. Denote $V \coloneqq V(\mathcal{H})$ and $E \coloneqq E(\mathcal{H})$. Let $c \coloneqq 4\sqrt{k}$. Fix some positive constant $z$ (to be determined later). Let $g \colon \mathbb{Z}_{\geq 1} \to \mathbb{R}$ be given by
		$$
			g(t) \coloneqq \begin{cases}
				1 - z^{-1} \text{ if } t = 1;\\
				2^{1-t}z^{-1} \text{ if } t > 1.	
			\end{cases}
		$$
		Inductively construct a sequence $\{V_i\}_{i = 0}^m$, where $0 \leq m \leq n$, of subsets of $V$ according to the following rule. Let $V_0 \coloneqq V$. If there is a vertex $v \in V_i$ such that
		\begin{equation}\label{eq:degree}
			\sum_{\substack{H \in E:\\ H \ni v}} g(|H \cap V_i|) \geq k - c,
		\end{equation}
		then select one such vertex, denote it by $v_i$, and let $V_{i+1} \coloneqq V_i \setminus \{v_i\}$. Otherwise let $m \coloneqq i$ and stop.

		If $m = n$, then
		$$
			|E| = \sum_{H \in E} 1 > \sum_{H \in E} \sum_{j = 1}^{|H|} g(j) = \sum_{i = 0}^{n-1} \sum_{\substack{H \in E:\\ H \ni v_i}} g(|H \cap V_i|) \geq (k-c) n,
		$$
		as desired.
		
		Now suppose that $m < n$. We will prove that this cannot happen. Let $V' \coloneqq V_m$. Since $V'$ is nonempty, the hypergraph $\mathcal{H} - V'$ obtained from $\mathcal{H}$ by deleting the vertices in $V'$ is $k$-colorable. Fix a proper $k$-coloring $\psi$ of $\mathcal{H}-V'$ and extend it to a $k$-coloring $\varphi$ of $\mathcal{H}$ by choosing a color for each vertex in $V'$ uniformly and independently from all other vertices.
		
		Let $A \subseteq \powerset{V'}$ be given by
		$$
			A \coloneqq \{ S \subseteq V' \,:\, \text{there is no $\varphi$-monochromatic edge $H \subseteq (V \setminus V') \cup S$}\}.
		$$
		Note that $A$ is downwards-closed and $\Pr(\emptyset \in A) = 1$ (because the coloring $\psi$ of $V \setminus V'$ is proper). We will use Theorem~\ref{theo:hypercubes} to prove that $\Pr(V' \in A) > 0$, which will be a contradiction since $\mathcal{H}$ is not $k$-colorable.
		
		For $v \in V'$, let
		$$
			\mathcal{B}(v) \coloneqq \{B_H \,:\, v \in H \in E\},
		$$
		where the event $B_H$ happens if and only if $H$ is $\varphi$-monochromatic. Clearly, if $v \in \partial A$, then at least one of the events $B_H \in \mathcal{B}(v)$ holds.
		
		Let $\tau(v) = \tau \in [1;+\infty)$ be a constant function. Consider some $B_H \in \mathcal{B}(v)$. There are two cases. First suppose that $H \not \subseteq V'$. Note that such $H$ is $\varphi$-monochromatic if and only if $H \setminus V'$ is $\psi$-monochromatic and $\varphi(u) = \psi(w)$ for all $u \in H \cap V'$ and $w \in H\setminus V'$. Therefore, for each such $H$ and for $Z \subseteq V' \setminus H$, $\Pr(B_H \vert Z \in A) \leq \Pr(B_H) \leq 1/k^{|H \cap V'|}$.
		Thus,
		$$
			\sigma^A_\tau(B_H, v) \leq \sigma^A_\tau(B_H, H\cap V') \leq \max_{Z \subseteq V' \setminus H} \Pr(B_H \vert Z \in A) \cdot \tau^{|H \cap V'|} \leq \frac{\tau^{|H \cap V'|}}{k^{|H \cap V'|}}.
		$$
		
		If, on the other hand, $H \subseteq V'$, then choose an arbitrary vertex $u \in H \setminus \{v\}$ and consider $Z \subseteq (V' \setminus H) \cup \{u\}$. (This idea is analogous to the one we discussed in Subsection~\ref{subsec:hypcol}.) Since fixing the color of $u$ does not change the probability that $H$ is monochromatic, we have $\Pr(B_H \vert Z \in A) \leq 1/k^{|H|-1}$, so
		$$
			\sigma^A_\tau(B_H, v) \leq \sigma^A_\tau(B_H, E \setminus \{u\}) \leq \max_{Z \subseteq (V' \setminus H) \cup \{u\}} \Pr(B_H \vert Z \in A) \cdot \tau^{|H \setminus \{u\}|} \leq \frac{\tau^{|H|-1}}{k^{|H|-1}}.
		$$
		
		For a vertex $v \in V'$, let
		$$
			a_t(v) \coloneqq |\{H \in E\,:\, v \in H \not\subseteq V', \,|H\cap V'| = t\}|;
		$$
		$$
			b_t(v) \coloneqq |\{H \in E\,:\, v \in H\subseteq V',\, |H| = t\}|.
		$$
		To apply Theorem~\ref{theo:hypercubes}, it is enough to guarantee that there exists a constant $\tau \in [1; +\infty)$ such that for all $v \in V'$,
		\begin{equation}\label{eq:requirement}
			\tau \geq 1 + \sum_{t = 1}^\infty a_t(v) \frac{\tau^t}{k^t} + \sum_{t = 3}^\infty b_t(v) \frac{\tau^{t-1}}{k^{t-1}}.
		\end{equation}
		
		Since $V'$ is the last set in the sequence $\{V_i\}_{i = 0}^m$, no vertex in $V'$ satisfies~\eqref{eq:degree}. In other words, for all $v \in V'$,
		\begin{equation}\label{eq:condition}
			\sum_{t = 1}^\infty a_t(v) g(t) + \sum_{t = 3}^\infty b_t(v) g(t) < k-c.
		\end{equation}
		Let
		$$
			\alpha_t(v) \coloneqq a_t(v)g(t);
		$$
		$$
			\beta_t(v) \coloneqq b_t(v)g(t).
		$$
		Then \eqref{eq:condition} can be rewritten as
		$$
			\gamma(v) \coloneqq \sum_{t=1}^\infty \alpha_t(v) + \sum_{t=3}^\infty \beta_t(v) < k-c,
		$$
		and \eqref{eq:requirement} turns into
		$$
			\tau \geq 1 + \sum_{t = 1}^\infty \alpha_t(v) \cdot \frac{1}{g(t)}\left(\frac{\tau}{k}\right)^t + \sum_{t = 3}^\infty \beta_t(v)\cdot \frac{1}{g(t)}\left(\frac{\tau}{k}\right)^{t-1},
		$$
		which, after substituting the actual values for $g$, becomes
		\begin{equation}\label{eq:requirement1}
			\tau \geq 1 + \alpha_1(v) \cdot \frac{z}{z-1} \frac{\tau}{k} + \sum_{t = 2}^\infty \alpha_t(v) \cdot \frac{1}{2}z\left(\frac{2\tau}{k}\right)^t + \sum_{t = 3}^\infty \beta_t(v) \cdot z\left(\frac{2\tau}{k}\right)^{t-1}.
		\end{equation}
		We can view the right-hand side of \eqref{eq:requirement1} as a linear combination of variables $\alpha_t(v)$, $\beta_t(v)$. If we assume that
		$$
			\frac{4\tau}{k} \geq \frac{1}{z-1},
		$$
		then the largest coefficient in this linear combination is $z \left(2\tau\middle/k\right)^2$ (the coefficient of $\beta_3(v)$). Thus, it is enough to find $\tau$, $z$ satisfying the following two inequalities:
		\begin{equation}\label{eq:req1}
			\frac{4\tau}{k} \geq \frac{1}{z-1};
		\end{equation}	
		\begin{equation}\label{eq:req2}
			\tau \geq 1 + \frac{4z\tau^2 (k-c)}{k^2}.
		\end{equation}
		(Inequality~\eqref{eq:req2} is obtained by replacing all coefficients on the right hand side of~\eqref{eq:requirement1} by the largest one and using the fact that $\gamma(v) < k-c$.) If we choose
		$$
			z = \frac{k}{4\tau}+1,
		$$
		then~\eqref{eq:req1} is satisfied, while~\eqref{eq:req2} becomes
		$$
			\tau \geq 1 + \frac{4\tau^2(k-c)}{k^2} \left(\frac{k}{4\tau}+1\right) = 1 + \frac{k-c}{k} \tau + \frac{4(k-c)}{k^2}\tau^2.
		$$
		Thus, we just have to make sure that the following inequality has a solution $\tau$:
		$$
			\frac{4(k-c)}{k^2}\tau^2 - \frac{c}{k} \tau + 1 \leq 0.
		$$
		This is true if and only if $c^2 \geq 16(k-c)$; in particular, $c = 4\sqrt{k}$ works. Therefore, $\varphi$ is a proper $k$-coloring of $\mathcal{H}$ with positive probability. This contradiction completes the proof.
	 \end{proof}

	\subsection{Choice functions}\label{subsec:choice}
	
	Our last example is a probabilistic corollary of the LCL. Let $U_1$, \ldots, $U_n$ be a collection of pairwise disjoint nonempty finite sets. A \emph{choice function} $F$ is a subset of $\bigcup_{i=1}^n U_i$ such that for all $1 \leq i \leq n$, $|F \cap U_i| = 1$. A \emph{partial choice function} $P$ is a subset of $\bigcup_{i=1}^n U_i$ such that for all $1 \leq i \leq n$, $|P \cap U_i| \leq 1$. For a partial choice function $P$, let
	$$
	\operatorname{dom}(P) \coloneqq \{i \,:\, P \cap U_i \neq \emptyset\}. 
	$$
	Thus, a choice function $F$ is a partial choice function with $\operatorname{dom}(F) = \{1, \ldots, n\}$.
	
	Let $F$ be a choice function and let $P$ be a partial choice function. We say that $P$ \emph{occurs} in $F$ if $P \subseteq F$, and we say that $F$ \emph{avoids} $P$ if $P$ does not occur in $F$. Many natural combinatorial problems (especially ones related to coloring) can be stated using the language of choice functions. For instance, consider a graph $G$ with vertex set $\{1, \ldots, n\}$. Fix a positive integer $k$ and let $U_i \coloneqq \{(i, c)\,:\, 1 \leq c \leq k\}$ for each $1 \leq i \leq n$. For each edge $ij \in E(G)$ and $1 \leq c \leq k$, define a partial choice function $P^c_{ij} \coloneqq \{(i, c), (j, c)\}$. Then a proper vertex $k$-coloring of $G$ can be identified with a choice function $F$ such that none of $\{P^c_{ij}\}_{ij \in E(G), 1\leq c \leq k}$ occur in $F$. Another problem that has a straightforward formulation using choice functions is the $k$-SAT (which also serves as a standard example of a problem that can be approached with the~LLL).
	
	A \emph{multichoice function} $M$ is simply a subset of $\bigcup_{i=1}^n U_i$ (one should think of it as a generalized choice function where one is allowed to choose multiple or zero elements from each set). For a multichoice function $M$, let $M_i \coloneqq M \cap U_i$. Again, we say that a partial choice function $P$ \emph{occurs} in a multichoice function $M$ if $P \subseteq M$. Suppose that we are given a family $P_1$, \ldots, $P_m$ of nonempty ``forbidden'' partial choice functions. For a multichoice function~$M$, the \emph{$i^\text{th}$ defect} of $M$ (notation: $\operatorname{def}_i(M)$) is the number of indices $j$ such that $i \in \operatorname{dom}(P_j)$ and $P_j$ occurs in $M$. Observe that there exists a choice function $F$ that avoids all of $P_1$, \ldots, $P_m$ if and only if there exists a multichoice function $M$ such that for all $1 \leq i \leq n$,
	\begin{equation}\label{eq:multi}
	|M_i| \geq 1 + \operatorname{def}_i(M).
	\end{equation}
	Indeed, if $F$ avoids all of $P_1$, \ldots, $P_m$, then $F$ itself satisfies \eqref{eq:multi}. On the other hand, if $M$ satisfies \eqref{eq:multi}, then, for every $i$, there is an element $x_i \in M_i$ that does not belong to any $P_j$ occurring in $M$. Therefore, $\{x_i\}_{i=1}^n$ is a choice function that avoids all of $P_1$, \ldots, $P_m$, as desired.
	
	The main result of this subsection is that, in fact, it is enough to establish \eqref{eq:multi} \emph{on average} for some random multichoice function $M$.
	
	\begin{theo}\label{theo:choice}
		Let $U_1$, \ldots, $U_n$ be a collection of pairwise disjoint nonempty finite sets and let $P_1$, \ldots, $P_m$ be a family of nonempty partial choice functions. Let $\Omega$ be a probability space and let $M_i \colon \Omega \to \powerset{U_i}$, $1 \leq i \leq n$, be a collection of mutually independent random variables. Set $M \coloneqq \bigcup_{i=1}^n M_i$. If for all $1 \leq i \leq n$,
		\begin{equation}\label{eq:multimean}
		\mathbb{E}|M_i| \geq 1 + \mathbb{E} \operatorname{def}_i(M),
		\end{equation}
		then there exists a choice function $F$ that avoids all of $P_1$, \ldots, $P_m$.
	\end{theo}
	\begin{proof}
		For $x \in \bigcup_{i=1}^n U_i$, let $p(x) \coloneqq \Pr(x \in M)$. Then
		$$
		\mathbb{E} |M_i| = \sum_{x \in U_i} p(x).
		$$
		Since the variables $\{M_i\}_{i=1}^n$ are independent,
		$$
		\Pr(P_j \subseteq M) = \prod_{x \in P_j} p(x).
		$$
		Therefore, if $N_i \coloneqq \{j \,:\, i \in \operatorname{dom}(P_j)\}$,
		$$
		\mathbb{E}\operatorname{def}_i(M) = \sum_{j \in N_i} \Pr(P_j \subseteq M) = \sum_{j \in N_i} \prod_{x \in P_j} p(x).
		$$
		Thus, \eqref{eq:multimean} is equivalent to
		\begin{equation}\label{eq:sums1}
		\sum_{x \in U_i} p(x) \geq 1 + \sum_{j \in N_i} \prod_{x \in P_j} p(x).
		\end{equation}
		Let $\tau(i) \coloneqq \sum_{x \in U_i} p(x)$ and let $q(x) \coloneqq p(x)/\tau(i)$ for all $x \in U_i$. Then \eqref{eq:sums1} can be rewritten as
		\begin{equation}\label{eq:sums2}
		\tau(i) \geq 1 + \sum_{j \in N_i} \prod_{x \in P_j} q(x) \cdot \tau(\operatorname{dom}(P_j)) .
		\end{equation}
		
		We will only use the numerical condition~\eqref{eq:sums2}, ignoring its probabilistic meaning. Construct a random choice function $F$ (in a new probability space) as follows: Choose an element $x \in U_i$ with probability $q(x)$, making the choices for different $U_i$'s independently (this definition is correct, since $\sum_{x \in U_i} q(x) = 1$). Set $I \coloneqq \{1, \ldots, n\}$ and define a random subset $A\subseteq \powerset{I}$ as follows:
		$$
			A \coloneqq \{S \subseteq I \,:\, \text{no $P_j$ with $\operatorname{dom}(P_j) \subseteq S$ occurs in $F$}\}.
		$$
		Then $A$ is a nonempty downwards-closed family of subsets of $I$, and $I \in A$ if and only if $F$ avoids all of $P_1$, \ldots, $P_m$.
		
		For $i \in I$, let
		$$
			\mathcal{B}(i) \coloneqq \{B_j \,:\, j \in N_i\},
		$$
		where the event $B_j$ happens if and only if $P_j \subseteq F$. Clearly, if $i \in \partial A$, then there is some $j \in N_i$ such that $P_j \subseteq F$, so we can apply Theorem~\ref{theo:hypercubes}.
		
		Consider any $i \in I$ and $j \in N_i$. Since $\Pr(B_j) = \prod_{x \in P_j} q(x)$, we have
		\begin{align*}
			\sigma^A_\tau(B_j, i) \leq \sigma^A_\tau(B_j, \operatorname{dom}(P_j)) &\leq \max_{Z \subseteq I \setminus \operatorname{dom}(P_j)}\Pr(B_j \vert Z \in A) \cdot \tau(\operatorname{dom}(P_j)) \\
			&\leq \Pr(B_j) \cdot \tau(\operatorname{dom}(P_j)) = \prod_{x \in P_j} q(x) \cdot \tau(\operatorname{dom}(P_j)).
		\end{align*}
		
		Therefore, in this case \eqref{eq:sums2} implies \eqref{eq:special_main}, yielding $\Pr\left(I \in A\right) > 0$, as desired.
	\end{proof}
	
	Theorem~\ref{theo:choice} can be used, for instance, to obtain condition \eqref{eq:hypcol} for $2$-colorability of uniform hypergraphs, or to prove that $a'(G) \leq \lceil 9.53(\Delta(G) - 1) \rceil$ (this bound, although considerably weaker than the one given by Theorem~\ref{theo:acyclic}, is still an improvement over the previous results derived using the~LLL). Another application of Theorem~\ref{theo:choice} can be found in~\cite{Bernshteyn2}.
	
	\paragraph{Acknowledgments.}
	This work is supported by the Illinois Distinguished Fellowship.
	I am grateful to Alexandr Kostochka for his helpful conversations and encouragement and to the anonymous referees for their valuable comments.

\end{document}